\documentclass[a4paper,12pt]{article}
\usepackage[utf8]{inputenc}

\usepackage[T1]{fontenc}

\usepackage{tikz}

\usepackage{subcaption}

\usepackage{tocloft}

\usepackage[english]{babel}
\usepackage{amsfonts,stmaryrd}
\usepackage{amssymb,amsthm}
\usepackage{xcolor,graphicx}
\usepackage{marginnote}
\usepackage{amsmath}
\usepackage{a4wide}
\usepackage[affil-it]{authblk}

\usepackage{epstopdf}
\newcommand{\doi}[1]{\href{http://dx.doi.org/#1}{doi:#1}}

\usepackage{titlesec}
\titlelabel{\thetitle.\quad}

\usepackage[labelsep=period]{caption}

\usepackage{enumitem}

\usepackage{crossreftools}
\makeatletter
\newcommand{\optionaldesc}[2]{%
	\phantomsection
	#1\protected@edef\@currentlabel{#1}\label{#2}%
}
\makeatother

\numberwithin{equation}{section}

\let\OLDthebibliography\thebibliography
\renewcommand\thebibliography[1]{
	\OLDthebibliography{#1}
	\setlength{\parskip}{1pt}
	\setlength{\itemsep}{1pt plus 0.3ex}
}

\usepackage[colorlinks=true,linktocpage,pdfpagelabels,
bookmarksnumbered,bookmarksopen]{hyperref}
\definecolor{ForestGreen}{rgb}{0.1,0.6,0.05}
\definecolor{EgyptBlue}{rgb}{0.063,0.1,0.6}
\hypersetup{
	colorlinks=true,
	linkcolor=EgyptBlue,         
	citecolor=ForestGreen,
	urlcolor=olive
}

\usepackage[hyperpageref]{backref}

\allowdisplaybreaks
\usepackage{mathtools}
\mathtoolsset{showonlyrefs}

\usepackage{orcidlink}

\usepackage{accents}

\def\Wo{W_0^{1,p}(\Omega)}

\newcommand{\boxbarl}{\scalebox{0.8}{$\boxbar$}}
\newcommand{\boxbslashl}{\scalebox{0.8}{$\boxbslash$}}

\newtheorem{theorem}{Theorem}[section]
\newtheorem{lemma}[theorem]{Lemma}
\newtheorem{proposition}[theorem]{Proposition}
\newtheorem{corollary}[theorem]{Corollary}
\theoremstyle{definition}
\newtheorem{remark}[theorem]{Remark}

\title{\vspace*{-5ex}
Non-Ljusternik--Schnirelman eigenvalues of the pure $p$-Laplacian exist
}

\author{Vladimir Bobkov\\}
\date{}

\AtEndDocument{%
	\par
	\medskip
	\bigskip
	\begin{tabular}{@{}l@{}}%
		(V.~Bobkov)\\[0.2em]
		\textsc{Institute of Mathematics, Ufa Federal Research Centre, RAS}\\
		\textsc{Chernyshevsky str. 112, 450008 Ufa, Russia}\\[0.3em]
		\orcidlink{0000-0002-4425-0218} 0000-0002-4425-0218\\[0.3em]
		\textit{E-mail address}: \texttt{bobkov@matem.anrb.ru, bobkovve@gmail.com}\\[1.5em]
\end{tabular}}

\begin{document}
	\maketitle
	\vspace*{-5ex}
	\begin{abstract}
		
		An old and well-known open problem in the critical point theory asks whether, for some $p \neq 2$ and some bounded domain $\Omega$, there exists a critical value of the $p$-Dirichlet energy $\|\nabla u\|_{p;\Omega}^p$ over an $L^p(\Omega)$-sphere in $W_0^{1,p}(\Omega)$ lying outside of a Ljusternik--Schnirelman type sequence of critical values, the latter will be called \textit{LS eigenvalues} of the $p$-Laplacian. 
		In this work, we provide a positive answer by showing the existence of a non-LS eigenvalue when $p>2$ is sufficiently close to $2$ and $\Omega$ is just a planar rectangle close to the square. 
		The arguments pursue the observation that a \textit{simple} eigenvalue of the Laplacian can be a common point for several branches of eigenvalues of the $p$-Laplacian as $p$ varies. 
		Since LS eigenvalues are continuous with respect to $p$ and exhaust the whole spectrum when $p=2$, we deduce that at least one of the branches must contain non-LS eigenvalues. 
		
		
	\end{abstract}

\section{Introduction}\label{sec:intro}

	Let $\Omega$ be a bounded domain in $\mathbb{R}^2$ and let $p \in (1,+\infty)$. 
	Consider the canonical eigenvalue problem for the Dirichlet $p$-Laplacian:
	\begin{equation}\label{eq:D}
		\tag{$\mathcal{D}$}
		\left\{
		\begin{aligned}
			-\text{div}(|\nabla u|^{p-2} \nabla u) &= \lambda |u|^{p-2} u 
			&&\text{in } \Omega, \\
			u &= 0  &&\text{on } \partial \Omega.
		\end{aligned}
		\right.
	\end{equation}	
	This problem is understood in the weak sense: a number $\lambda \in \mathbb{R}$ and a function $u \in \Wo \setminus \{0\}$ are called \textit{eigenvalue} and \textit{eigenfunction} of the Dirichlet $p$-Laplacian in $\Omega$, respectively, provided 
	\begin{equation}\label{eq:D:weak}
		\int_\Omega |\nabla u|^{p-2} \langle \nabla u, \nabla \phi \rangle \,dxdy
		= 
		\lambda \int_\Omega |u|^{p-2} u \phi \,dxdy
		\quad \text{for all}~ \phi \in \Wo.
	\end{equation}
	Thus, eigenvalues coincide with critical values of the $p$-Dirichlet energy $\|\nabla u\|_{p;\Omega}^p$ over the unit $L^p(\Omega)$-sphere in $W_0^{1,p}(\Omega)$, and eigenfunctions are associated with critical points. 
	In view of the homogeneity reasons, the problem of finding eigenvalues is equivalent to finding critical values of the Rayleigh quotient 
	$$
	R_p[u; \Omega] := \frac{\int_\Omega |\nabla u|^p \,dxdy}{\int_\Omega |u|^p \,dxdy}
	\quad \text{over}~ u \in \Wo \setminus \{0\}. 
	$$
	In the linear case $p=2$, the standard Courant--Fisher minimax formula
	\begin{equation}\label{eq:eigenvalue_linear}
	\lambda_{k}(2;\Omega) 
	:= 
	\min_{X_k} \max_{u \in X_k \setminus \{0\}} R_2[u; \Omega], 
	\quad k \in \mathbb{N},
	\end{equation}
	characterizes the whole spectrum, i.e., the full set of critical values including multiplicities. 
	Here, the minimum is taken over subspaces $X_k \subset W_0^{1,2}(\Omega)$ of dimension $k$. 
	In the nonlinear case $p\neq 2$, the same construction does not produce critical values, in general. 
	Instead, an infinite sequence of critical values can be described by means of the minimax variational principle of Ljusternik--Schnirelman type.
	As a particular and arguably most popular case, let us introduce it via the Krasnoselskii genus, and we refer to Remark~\ref{reg:gen} for a discussion on other options. 
	For a symmetric set $\mathcal{A} \subset W^{1,p}_0(\Omega)$, that is, $\mathcal{A}=-\mathcal{A}$, 
	we define the Krasnoselskii genus of $\mathcal{A}$ as
	$$ 
	\gamma_K(\mathcal{A}) 
	=
	\inf\{m\in\mathbb{N}:~ \exists \mbox{ a continuous odd map } f:\mathcal{A} \to \mathbb{R}^{m} \setminus \{0\}\},
	$$
	with the convention $\gamma_K(\mathcal{A}) = +\infty$ if no required map $f$ exists for any $m \in \mathbb{N}$. 
	For $k \in \mathbb{N}$, we define
	$$ 
	\Sigma_k(p;\Omega) 
	= 
	\left\{\mathcal{A} \subset \Wo \setminus \{0\}:~  \mathcal{A} \mbox{ is symmetric, compact, and } \gamma_K(\mathcal{A})\geq k\right\}
	$$
	and
	\begin{equation}\label{highereigenvalues} 
	\lambda_k(p;\Omega) 
	= 
	\inf_{\mathcal{A} \in \Sigma_k(p;\Omega)} \max_{u \in \mathcal{A}} R_p[u;\Omega].  
	\end{equation}
	It is known that each $\lambda_k(p;\Omega)$ is indeed a critical value of $R_p[\cdot;\Omega]$, $\lambda_1(p;\Omega)$ and $\lambda_2(p;\Omega)$ are the first two minimal critical values, $\lambda_1(p;\Omega)$ is simple, 
	\begin{equation}\label{eq:sequenceoflambdas}
	0 < \lambda_1(p;\Omega) < \lambda_2(p;\Omega) \leq \dots \leq \lambda_k(p;\Omega) \to +\infty
	\quad \text{as } k \to +\infty,
	\end{equation}	
	and $\{\lambda_k(p;\Omega)\}$ obeys two-sided Weyl type bounds. 
	Moreover, if $\lambda_k(p;\Omega)$ has a multiplicity in the sense that $\lambda_k(p;\Omega)=\dots=\lambda_{k+l}(p;\Omega)$, then the Krasnoselskii genus of the set of all corresponding normalized eigenfunctions is not less than $l+1$. We refer to \cite{anane,cuesta,F,garciaperal}, as well as to general seminal works \cite{amann,browder1,FNSS,szulkin}. 
	Hereinafter, we call the critical values of the Ljusternik--Schnirelman type \eqref{highereigenvalues} as \textit{LS eigenvalues}.\footnote{Sometimes, LS eigenvalues are called ``variational'', see, e.g, \cite{bindingrynne1}, but, as fairly noted in \cite{drabek-var}, this notion might be ambiguous, as some ``nonvariational'' eigenvalues, whenever exist, could still be found by minimax variational methods, see \cite{drabek-takac}. We also refer to Remark~\ref{rem:variational} below.}
	
	At the same time, the global structure of the spectrum (i.e., the set of all eigenvalues) of the $p$-Laplacian is highly obscure. 
	In particular, neither its discreteness nor even a countability is known. 
	In the present work, we are interested in the following problem in this regard: 
\begin{enumerate}[label={($\mathcal{P}$)}]
\item\label{Q} 
Is there any non-LS eigenvalue $\lambda$ of the $p$-Laplacian in $\Omega$, that is, $\lambda \not\in \{\lambda_k(p;\Omega)\}$? 
\end{enumerate}

The problem \ref{Q} is dated back not later than to the works of \textsc{Fu\v{c}\'ik, Ne\v{c}as,  Sou\v{c}ek, \& Sou\v{c}ek} (see, e.g., \cite{FNSS2} and the famous monograph \cite{FNSS}), who were interested in finding not only lower bounds, but also upper bounds on the number of critical points of nonlinear functionals coming from the PDE theory. 
In particular, under sufficiently strong  assumptions, they succeeded to prove that the set of critical values of general nonlinear functionals related to the present settings is countable \cite{FNSS2}. 
Some examples showing that the Ljusternik--Schnirelman procedure does not describe all critical values of special  nonlinear finite dimensional mappings were also provided, see \cite[Remark, p.~115]{FNSS}. 
We also refer to \cite{anane} and \cite{F} for some of the earliest \textit{explicit} statements of the problem \ref{Q}.

It is known that the problem~\ref{Q} has a \textit{negative} answer when either $p=2$ or $N=1$, see, e.g., \cite[Theorem~3 e)]{ananechakrone}, \cite[Theorem~A.2]{brasco-parini-squassina}, \cite[Propositions~4.7]{cuesta} about $p=2$, and \cite[Theorem~4]{ananechakrone}, \cite[Propositions~4.6]{cuesta}, \cite[Lemma~3]{drabek-robinson} about $N=1$. 
In particular, the sequence \eqref{highereigenvalues} at $p=2$ agrees with \eqref{eq:eigenvalue_linear}. 
On the other hand, in contexts closely related to the present settings, some \textit{positive} answers are also known:
\begin{enumerate}
\item \textsc{Binding \& Rynne}
\cite{bindingrynne1} were probably the first who obtained the \textit{existence} of non-LS eigenvalues for the $p$-Laplacian. For a solid torus $\Omega$ in $\mathbb{R}^N$, $N \geq 2$, they proved that for any $p \neq 2$ and $n \in \mathbb{N}$ there exists a weight function $r \in \mathcal{C}^1(\overline{\Omega})$ such that the $p$-Dirichlet energy $\|\nabla u\|_{p;\Omega}^p$ over the ``weighted'' unit $L^p(\Omega;r)$-sphere in $W^{1,p}(\Omega)$ has at least $n$ non-LS eigenvalues, see \cite[Theorem~7.1]{bindingrynne1}. 
Notice that the underlying space is $W^{1,p}(\Omega)$ (not $W_0^{1,p}(\Omega)$), so eigenfunctions satisfy the Neumann boundary conditions. 
In the one-dimensional case, the existence of non-LS eigenvalues in a similar weighted case is obtained for the periodic boundary conditions, see \cite[Section~4.2]{bindingrynne1}.
The results also hold for the problem perturbed by a potential rather than by a weight. 

\item \textsc{Dr\'abek \& Tak\'a\v{c}} 
\cite{drabek-takac} worked in settings close to \cite{bindingrynne1} and shown that, in the one-dimensional case with the periodic boundary conditions and a potential, there exists a non-LS eigenvalue characterized by a minimax formula, see \cite[Theorem~1.1]{drabek-takac}.
The minimization set in this formula consists of paths locating in a \textit{particular} subset of $W^{1,p}(\Omega)$, in contrast to the set $\Sigma_2(p;\Omega)$ containing \textit{all} symmetric closed paths in $W^{1,p}(\Omega)$.

\item \textsc{Brasco \& Franzina} \cite{brasco-franzina-pathalogical} proved that already when $\Omega$ is the union of two disjoint balls of dimension $N \geq 1$, for any $1<q<p$ the set of critical values of the $p$-Dirichlet energy $\|\nabla u\|_{p;\Omega}^p$ over the unit $L^q(\Omega)$-sphere in $W_0^{1,p}(\Omega)$ is \textit{not discrete} and has countably many accumulation points, so that there are plenty of non-LS eigenvalues, see \cite[Theorem~3.1]{brasco-franzina-pathalogical}.  
The subhomogeneity assumption $q<p$ seems vital, but it remains unknown whether the set of critical values might have a similar degenerate behavior when $\Omega$ is connected. 
\end{enumerate}

Nonetheless, to the best of our knowledge, there were no answer to the problem \ref{Q} in the ``pure'' $p$-Laplacian settings -- without weights or potentials, for the Dirichlet boundary conditions, in a connected $\Omega$, and with the same homogeneity ($q=p$). 
The aim of the present work is to justify the existence of a non-LS eigenvalue in this case, and hence to answer \ref{Q} \textit{affirmatively}. 

\medskip
Hereinafter, we work with planar rectangles $\mathcal{R}_a = (0,a) \times (0,1)$, where $a>0$.
We denote by $\lambda_{\boxbarl}(p; \mathcal{R}_{a})$ an eigenvalue of the $p$-Laplacian in $\mathcal{R}_{a}$ whose eigenfunction has opposite constant signs in two subrectangles symmetric with respect to the vertical middle line $\{x=a/2\}$ of $\mathcal{R}_{a}$. 
In the same manner, we denote by $\lambda_{\boxminus}(p; \mathcal{R}_{a})$ an eigenvalue of the $p$-Laplacian in $\mathcal{R}_{a}$ whose eigenfunction has opposite constant signs in two subrectangles symmetric with respect to the horizontal middle line $\{y=1/2\}$ of $\mathcal{R}_{a}$. 

Our main result is as follows. 
\begin{theorem}\label{thm:main}
For any $p>2$ sufficiently close to $2$, there exists $a^*>1$ such that for any $a \in (1,a^*)$ at least one of $\lambda_{\boxbarl}(p; \mathcal{R}_{a})$ and $\lambda_{\boxminus}(p; \mathcal{R}_{a})$ is a non-LS eigenvalue of the $p$-Laplacian in $\mathcal{R}_a$.
\end{theorem}

The arguments are based on the investigation of the mappings 
$$
(p,a) \mapsto \lambda_{k}(p; \mathcal{R}_{a}), 
\quad 
(p,a) \mapsto \lambda_{\boxbarl}(p; \mathcal{R}_{a})
\quad
(p,a) \mapsto \lambda_{\boxminus}(p; \mathcal{R}_{a})
$$
in a neighborhood of the point $(p,a)=(2,1)$, which corresponds to the Laplacian in the square. 
At this point we have
\begin{equation}\label{eq:chain0}
\lambda_2(2;\mathcal{R}_1)
=
\lambda_{\boxbarl}(2;\mathcal{R}_1)
=
\lambda_{\boxminus}(2;\mathcal{R}_1)
=
\lambda_3(2;\mathcal{R}_1)
<
\lambda_4(2;\mathcal{R}_1).
\end{equation}
However, we show, in essence, that in the claimed range of $p$ and $a$ the equalities in \eqref{eq:chain0} split as follows:
\begin{equation}\label{eq:chain1}
\lambda_2(p;\mathcal{R}_a)
<
\lambda_{\boxbarl}(p;\mathcal{R}_a)
<
\lambda_{\boxminus}(p;\mathcal{R}_a)
<
\lambda_4(p;\mathcal{R}_a),
\end{equation}
which means that $\lambda_{\boxbarl}(p;\mathcal{R}_a)$ or $\lambda_{\boxminus}(p;\mathcal{R}_a)$ does not coincide with $\lambda_{3}(p;\mathcal{R}_a)$. 
The first inequality in \eqref{eq:chain1} and the convergence of the corresponding mappings to $\lambda_2(2;\mathcal{R}_a)$ as $p \to 2$ further imply that a simple eigenvalue of the Laplacian in $\mathcal{R}_a$, $a>1$, can be a common point for several branches of eigenvalues of the $p$-Laplacian in $\mathcal{R}_a$ as $p$ varies.

\medskip
The work has the following structure. 
In Section~\ref{sec:preliminaries}, we provide some auxiliary results. 
The proof of Theorem~\ref{thm:main} is given in Section~\ref{sec:proof}, with a supplement of a few additional general remarks. 
Finally, in Section~\ref{sec:abitmoreaboutsquare}, we discuss a curious side observation on the nodal geometry of third LS-eigenfunctions of the $p$-Laplacian in the square $\mathcal{R}_1$.

\section{Preliminaries}\label{sec:preliminaries}

In this section, we collect several notation and auxiliary results needed for the proof of Theorem~\ref{thm:main}. 
By the definition of $\lambda_{\boxbarl}(p;\mathcal{R}_a)$ and the fact that the only sign-constant eigenfunction of the $p$-Laplacian is the first eigenfunction (see, e.g., \cite{kawohlling}), we have 
\begin{equation}\label{eq:lambda=lambda}
\lambda_{\boxbarl}(p;\mathcal{R}_a)
=
\lambda_{1}(p;\mathcal{R}_{a/2}).
\end{equation}
Occasionally, we work with the expanded notation $\mathcal{R}_a^b = (0,a) \times (0,b)$. In particular,  $\mathcal{R}_a^1 \equiv \mathcal{R}_a$. 
By the scaling property of eigenvalues, we have, as above,
\begin{equation}\label{eq:lambda=lambda2}
\lambda_{\boxminus}(p;\mathcal{R}_a)
=
\lambda_{1}(p;\mathcal{R}_{a}^{1/2})
=
2^p \lambda_{1}(p;\mathcal{R}_{2a}).
\end{equation}
Also, we denote by $\mathcal{T}_1$ the triangular domain spanned on the vertices $(0,0)$, $(1,0)$, $(0,1)$. 
Let us introduce the eigenvalue 
$\lambda_{\boxbslashl}(p;\mathcal{R}_1)$
of the $p$-Laplacian in the square $\mathcal{R}_1$ whose eigenfunction has opposite constant signs in $\mathcal{T}_1$ and its reflection with respect to the diagonal $\{x=-y\}$ of $\mathcal{R}_{1}$. 
Referring to \cite{kawohlling} as above, we have the following characterization:
\begin{equation}\label{eq:lambda=lambda3}
\lambda_{\boxbslashl}(p;\mathcal{R}_1)
=
\lambda_{1}(p;\mathcal{T}_1).
\end{equation}

For a function $u \in W_0^{1,p}(\mathcal{R}_a)$, we denote by $\tilde{u}: \mathbb{R}^2 \to \mathbb{R}$ the antisymmetric (odd) extension of $u$ with respect to the sides of $\mathcal{R}_a$. 
It can be defined a.e.\ as
$$
\tilde{u}(x,y)
=
\text{sgn}(x-2ka)\,\text{sgn}(y-2l)\,
u\bigl(|x-2ka|,|y-2l|\bigr),
$$
whenever $|x-2ka| < a$ and $|y-2l| < 1$ for some $k,l\in\mathbb Z$.
It is clear that 
$\tilde{u} \in W_0^{1,p}(\mathcal{R}_{ka}^k)$ for any $k \in \mathbb{N}$, and hence $\tilde{u} \in W^{1,p}_\text{loc}(\mathbb{R}^2)$. 

By $\|u\|_{q;\Omega}$ we denote the $L^q(\Omega)$-norm of $u$, for $q \in [1,+\infty]$. 
The notation $[u]_{1,\beta;\Omega}$ is used for the $(1,\beta)$-H\"older seminorm of $u$ over $\Omega$. 

\subsection{Uniform regularity of eigenfunctions}

\begin{lemma}\label{lem:regularity}
For any $1<p_1 \leq p_2$,  $0<a_1 \leq a_2$, and $0<b_1 \leq b_2$, there exist $\beta_* \in (0,1)$ and $C_*>0$ such that for any $p \in [p_1,p_2]$, $a \in [a_1,a_2]$, and for any eigenfunction $u$ of the $p$-Laplacian in $\mathcal{R}_a$ normalized as $\|u\|_{p;\mathcal{R}_a}=1$ and whose eigenvalue belongs to $[b_1,b_2]$, we have $u \in \mathcal{C}^{1,\beta_*}(\overline{\mathcal{R}_{a}})$ with $[u]_{1,\beta_*;\mathcal{R}_a} \leq C_*$. 
\end{lemma}
\begin{proof}
Applying a uniform $L^\infty$-estimate of
\cite[Lemma~4.1]{lind} to positive and negative parts of $u$, we obtain the existence of $C_1>0$ such that $\|u\|_{\infty;\mathcal{R}_a} \leq C_1$ for any $p \in [p_1,p_2]$, $a \in [a_1,a_2]$, and any normalized eigenfunction $u$ of the $p$-Laplacian in $\mathcal{R}_a$ whose eigenvalue $\lambda(p;\mathcal{R}_{a})$ belongs to $[b_1,b_2]$. 
It is easy to show (cf.\ \cite[Theorem~1.2]{ADS}) that the antisymmetric extension $\tilde{u} \in W^{1,p}_\text{loc}(\mathbb{R}^2)$ of $u$ weakly satisfies
\begin{equation}\label{eq:utilde}
-\Delta_{p} \tilde{u} = \lambda(p;\mathcal{R}_{a}) |\tilde{u}|^{p-2} \tilde{u}
\quad \text{in}~ \mathbb{R}^2. 
\end{equation}
Noting that $\|\tilde{u}\|_{\infty;\mathbb{R}^2} = \|u\|_{\infty;\mathcal{R}_a}$, fixing any compact set $K$ such that $\mathcal{R}_{a} \subset K$ for all $a \in [a_1,a_2]$, and applying \cite{tolksdorf}, we get the existence of 
$$
C_2 = C_2(C_1,p,K)>0 
\quad \text{and}\quad
\beta = \beta(C_1,p,K) \in (0,1)
$$
such that $u \in \mathcal{C}^{1,\beta}(K)$ and $[u]_{1,\beta;K} \leq C_2$. 
As noted in the proof of \cite[Theorem~6.3]{lind}, the dependence of $C_2$ and $\beta$ on $p$ in the arguments of \cite{tolksdorf} is of continuous nature, so 
$$
C_* := \sup_{p \in [p_1,p_2]} C_2 < +\infty
\quad \text{and}\quad
\beta_*:=\inf_{p \in [p_1,p_2]} \beta > 0
$$ 
for a fixed $K$. 
Hence, we get $u \in \mathcal{C}^{1,\beta_*}(\overline{\mathcal{R}_{a}})$ and $[u]_{1,\beta_*;\mathcal{R}_{a}} \leq C_*$, where $\beta_*$ and $C_*$ do not depend on $p \in [p_1,p_2]$, $a \in [a_1,a_2]$, and a normalized eigenfunction $u$ associated with $\lambda(p;\mathcal{R}_{a}) \in [b_1,b_2]$. 
\end{proof}

Using Lemma~\ref{lem:regularity} and 
the Arzel\`a-Ascoli theorem, we obtain the following convergence result. 
\begin{lemma}\label{lem:compactness}
Let $\{p_n\} \subset (1,+\infty)$ converge to $p > 1$. 
Let $\{a_n\} \subset (0,+\infty)$ converge to $a > 0$.
Let $\{\lambda_n\} \subset (0,+\infty)$ be a sequence of eigenvalues of the $p_n$-Laplacian in $\mathcal{R}_{a_n}$ converging to some $\lambda > 0$. 
For each $n$, let $u_n \in W_0^{1,p_n}(\mathcal{R}_{a_n})$ be an eigenfunction corresponding to $\lambda_n$ and normalized as $\|u_n\|_{p_n;\mathcal{R}_{a_n}} = 1$.
Then there exists an eigenfunction $u \in W_0^{1,p}(\mathcal{R}_a)$ of the $p$-Laplacian in $\mathcal{R}_a$ associated with the eigenvalue $\lambda$ such that $\|u\|_{p;\mathcal{R}_{a}}=1$ and 
$\tilde{u}_n \to \tilde{u}$ in $\mathcal{C}^1(\overline{\mathcal{R}_{2a}})$, up to a subsequence.
\end{lemma}

\subsection{Regularity of eigenvalues}

\begin{lemma}\label{lem:continuity}
Let $k \in \mathbb{N}$. 
Then the mapping $(p,a) \mapsto \lambda_k(p;\mathcal{R}_a)$ 
is continuous in $(1,+\infty) \times (0,+\infty)$. 
\end{lemma}
\begin{proof}
Since $\mathcal{R}_a$ is a Lipschitz domain, 
the continuity of $p \mapsto \lambda_k(p;\mathcal{R}_a)$ is established by \textsc{Parini} \cite{parini-continuity}. 
Alternatively, the claim can be derived from \cite{degiovannimarzocchi}: Corollary~6.2 states that the continuity of $p \mapsto \lambda_k(p;\mathcal{R}_a)$ for $k \in \mathbb{N}$ is equivalent to the left-continuity of $p \mapsto \lambda_1(p;\mathcal{R}_a)$; by Theorem 4.1 (a), (c) the required left-continuity is equivalent to the equality 
\begin{equation}\label{eq:sobolveq}
\bigcap_{1<s<p} \left(W^{1,p}(\mathcal{R}_a) \cap W_0^{1,s}(\mathcal{R}_a)\right)
=
W_0^{1,p}(\mathcal{R}_a);
\end{equation}
and, finally, Proposition~2.1 (e) justifies \eqref{eq:sobolveq} since $\mathcal{R}_a$ is Lipschitz and hence has the segment property. 
Both \cite{degiovannimarzocchi} and \cite{parini-continuity} are based on the work \cite{champione}. 

Let us now consider the continuity of 
$a \mapsto \lambda_k(p;\mathcal{R}_a)$. 
Note that the domain monotonicity of eigenvalues implies that $a \mapsto \lambda_k(p;\mathcal{R}_a)$ is nonincreasing, and hence it is locally bounded. 
For any $a_2>a_1>0$, we have 
$$
\mathcal{R}_{a_1}^{a_1/a_2}
\subset 
\mathcal{R}_{a_1}
\subset 
\mathcal{R}_{a_2},
$$
where we recall the notation $\mathcal{R}_{a_1}^{a_1/a_2} =(0,a_1) \times (0,a_1/a_2)$. 
Therefore, using the domain monotonicity and scaling property of $\lambda_k(p;\mathcal{R}_{a_1})$, we get
$$
\left(\frac{a_2}{a_1}\right)^{p_1} \lambda_k(p;\mathcal{R}_{a_2})
=
\lambda_k(p;\mathcal{R}_{a_1}^{a_1/a_2})
\geq
\lambda_k(p;\mathcal{R}_{a_1})
\geq
\lambda_k(p;\mathcal{R}_{a_2}),
$$
and hence
\begin{equation}\label{eq:contni1}
0 \leq 
\lambda_k(p;\mathcal{R}_{a_1})
-
\lambda_k(p;\mathcal{R}_{a_2})
\leq
\left(\left(\frac{a_2}{a_1}\right)^p-1\right)
\lambda_k(p;\mathcal{R}_{a_2})
\leq
\left(\left(\frac{a_2}{a_1}\right)^p-1\right)
\lambda_k(p;\mathcal{R}_{a_1}).
\end{equation}
From here the continuity of $a \mapsto \lambda_k(p;\mathcal{R}_a)$ follows.

Finally, we discuss the continuity of 
$(p,a) \mapsto \lambda_k(p;\mathcal{R}_a)$. 
Let us fix some $p_1>1$ and $a_1>0$. 
For any $p_2>1$ and $a_2>a_1$, using the triangle inequality and \eqref{eq:contni1}, we get
\begin{align}
|\lambda_k(p_1;\mathcal{R}_{a_1})
-
\lambda_k(p_2;\mathcal{R}_{a_2})|
&\leq
|\lambda_k(p_1;\mathcal{R}_{a_1})
-
\lambda_k(p_2;\mathcal{R}_{a_1})|
+
|\lambda_k(p_2;\mathcal{R}_{a_1})
-
\lambda_k(p_2;\mathcal{R}_{a_2})|
\\
&\leq
\label{eq:lem:cont:proof:1}
|\lambda_k(p_1;\mathcal{R}_{a_1})
-
\lambda_k(p_2;\mathcal{R}_{a_1})|
+
\left(\left(\frac{a_2}{a_1}\right)^{p_2}-1\right)
\lambda_k(p_2;\mathcal{R}_{a_1}).~~~
\end{align}
In view of the continuity of $p \mapsto \lambda_k(p;\mathcal{R}_a)$, for any $\varepsilon>0$ there exists $\delta>0$ such that $|\lambda_k(p_1;\mathcal{R}_{a_1})
-
\lambda_k(p_2;\mathcal{R}_{a_1})| < \varepsilon$ whenever $p_2 \in (p_1-\delta,p_1+\delta)$.
Consequently, for such $p_2$,  
\eqref{eq:lem:cont:proof:1} yields
$$
|\lambda_k(p_1;\mathcal{R}_{a_1})
-
\lambda_k(p_2;\mathcal{R}_{a_2})|
< 
\varepsilon
+
\left(\left(\frac{a_2}{a_1}\right)^{p_1+\delta} - 1\right)
(\varepsilon + \lambda_k(p_1;\mathcal{R}_{a_1})).
$$
Decreasing $\delta>0$ so that $\left(\left(a_2/a_1\right)^{p_1+\delta}-1\right)
(\varepsilon + \lambda_k(p_1;\mathcal{R}_{a_1})) < \varepsilon$ for any $a_2 \in (a_1,a_1+\delta)$, we arrive at $|\lambda_k(p_1;\mathcal{R}_{a_1})
-
\lambda_k(p_2;\mathcal{R}_{a_2})|<2\varepsilon$. 
Changing the roles of $a_1$ and $a_2$, we obtain an analogous estimate for $a_2 \in (a_1-\delta,a_1)$. 
This results in the desired continuity of the mapping $(p,a) \mapsto \lambda_k(p;\mathcal{R}_a)$. 
\end{proof}

\begin{lemma}\label{lem:differentiability:l1}
The mapping $(p,a) \mapsto \lambda_1(p;\mathcal{R}_a)$ is continuously differentiable in $(1,+\infty) \times (0,+\infty)$, and its partial derivatives satisfy
\begin{align}\label{eq:lem:differentiability:l11}
\partial_p \lambda_1(p;\mathcal{R}_a)
&=
\int_{\mathcal{R}_a} |\nabla u|^p \ln |\nabla u| \,dxdy
-
\,
\lambda_1(p;\mathcal{R}_a)
\int_{\mathcal{R}_a} |u|^p \ln |u| \,dxdy,\\
\label{eq:lem:differentiability:l12}
\partial_a \lambda_1(p;\mathcal{R}_a)
&=
-(p-1)
\int_0^1 |\partial_x u(a,y)|^p \,dy,
\end{align}
where $u$ is the first eigenfunction of the $p$-Laplacian in $\mathcal{R}_a$ normalized as $\|u\|_{p;\mathcal{R}_a}=1$. 
\end{lemma}
\begin{proof}
The result of \textsc{Barbatis \& Lamberti} \cite[Theorem~9]{barbatis-lamberti} says that 
if $\Omega$ is a bounded $\mathcal{C}^{1,\alpha}$-domain for some $\alpha \in (0,1)$, then $p \mapsto \lambda_1(p;\Omega)$  is continuously differentiable and 
\begin{equation}\label{eq:lambdapprime}
\partial_p \lambda_1(p;\Omega)
=
\int_\Omega |\nabla v|^p \ln |\nabla v| \,dxdy
-
\, 
\lambda_1(p;\Omega)
\int_\Omega |v|^p \ln |v| \,dxdy,
\end{equation}
where $v$ is the corresponding first eigenfunction normalized as $\|v\|_{p;\Omega}=1$.\footnote{Note that the formula from \cite[Theorem~9]{barbatis-lamberti} contains an additional multiple $p$ for both terms on the right-hand side. This factor, however, is erroneous: it originates from \cite[Lemma~4]{barbatis-lamberti} containing a redundant factor $s$.}
In our settings, $\Omega$ is a rectangle, which does not have the required $\mathcal{C}^{1,\alpha}$-regularity.
However, the inspection of the proof of \cite[Theorem~9]{barbatis-lamberti} shows that the uniform $\mathcal{C}^{1,\beta_*}(\overline{\mathcal{R}_{a}})$-regularity of eigenfunctions given by Lemma~\ref{lem:regularity} is sufficient for the validity of \eqref{eq:lambdapprime}, and hence \eqref{eq:lem:differentiability:l11} holds. 

The result of \textsc{Garc\'ia Meli\'an \& Sabina de Lis} \cite{garcia} (see also \cite{diblasio-lamberti}) says that if $\Omega$ is a $\mathcal{C}^{2,\alpha}$-domain for some $\alpha \in (0,1)$ and if $T_\delta \in \mathcal{C}^1(\overline{\Omega},\mathbb{R}^2)$ is a mapping of the form
$$
T_\delta(x,y) = (x,y) + \delta R(x,y),
$$
where $R \in \mathcal{C}^1(\overline{\Omega},\mathbb{R}^2)$, then $\delta \mapsto \lambda_1(p;T_\delta(\Omega))$ is differentiable at zero and the following Hadamard shape derivative formula holds: 
\begin{equation}\label{eq:lambdapprime3}
\left.
\partial_\delta
\lambda_1(p;T_\delta(\Omega))
\right|_{\delta=0}
=
-(p-1)
\int_{\partial\Omega} |\langle \nabla v, \nu\rangle|^p \langle R, \nu \rangle \,d\sigma,
\end{equation}
where the first eigenfunction $v$ is normalized as $\|v\|_{p;\Omega}=1$ and $\nu$ is the unit outward normal vector to $\partial\Omega$. 
The regularity is imposed on $\Omega$ mainly to have the well-defined normal vector, to approximate $v$ in the $\mathcal{C}^1(\overline{\Omega})$-topology by a sequence $\{v_\varepsilon\} \subset \mathcal{C}^{2,\alpha}(\overline{\Omega})$ of solutions of a family of regularized problems, and to rigorously perform the integration by parts for each $v_\varepsilon$. 
In the present settings, although $\Omega$ is not $\mathcal{C}^{2,\alpha}$-smooth, its rectangular shape allows to work with the antisymmetric extension $\tilde{v}$ of $v$, regularize $\tilde{v}$ in a way similar to  \cite{garcia} (see also \cite{diblasio-lamberti}) on a larger rectangle (cf.\ \eqref{eq:utilde}), and then restrict back to $\mathcal{R}_a$. 
We omit technical details. 
Taking $R(x,y) = (x/a,0)$, we have  $T_\delta(\mathcal{R}_a) = \mathcal{R}_{a+\delta}$. 
With this choice, \eqref{eq:lambdapprime3} turns to \eqref{eq:lem:differentiability:l12}. 

Thanks to the formulas \eqref{eq:lem:differentiability:l11}, \eqref{eq:lem:differentiability:l12}, 
the continuity of the function $t \mapsto |t|^p \ln |t|$, and the fact that the only sign-constant eigenfunction of the $p$-Laplacian is the first eigenfunction (see, e.g., \cite{kawohlling}), we deduce from Lemma~\ref{lem:compactness} that the mappings $(p,a) \mapsto \partial_p \lambda_1(p;\mathcal{R}_a)$ and $(p,a) \mapsto \partial_a \lambda_1(p;\mathcal{R}_a)$ are continuous in $(1,+\infty) \times (0,+\infty)$.
This gives the desired continuous differentiability of $\lambda_1(p;\mathcal{R}_a)$ with respect to $p$ and $a$. 
\end{proof}

\begin{remark}\label{rem:differentiability:ltriangle}
Arguing as in Lemma~\ref{lem:differentiability:l1}, we obtain that the mapping $p \mapsto \lambda_1(p;\mathcal{T}_1)$ is also continuously differentiable in $(1,+\infty)$ and 
\begin{equation}\label{eq:lem:differentiability:ltriangle}
\partial_p \lambda_1(p;\mathcal{T}_1)
=
\int_{\mathcal{T}_1} |\nabla u|^p \ln |\nabla u| \,dxdy
-
\,
\lambda_1(p;\mathcal{T}_1)
\int_{\mathcal{T}_1} |u|^p \ln |u| \,dxdy,
\end{equation}
where $u$ is the first eigenfunction of the $p$-Laplacian in $\mathcal{T}_1$ with  $\|u\|_{p;\mathcal{T}_1}=1$. 
\end{remark}

Recalling \eqref{eq:lambda=lambda}, \eqref{eq:lambda=lambda2}, and \eqref{eq:lambda=lambda3}, we obtain the following corollary.
\begin{corollary}\label{cor:differentiability}
The mappings $(p,a) \mapsto \lambda_{\boxbarl}(p;\mathcal{R}_a)$ and $(p,a) \mapsto \lambda_{\boxminus}(p;\mathcal{R}_a)$ are continuously differentiable in $(1,+\infty) \times (0,+\infty)$. 
Moreover, the mapping $p \mapsto \lambda_{\boxbslashl}(p;\mathcal{R}_1)$ is continuously differentiable in $(1,+\infty)$.
\end{corollary}

\subsection{Properties of eigenvalues in the square}\label{sec:propsquare}
It is well-known that 
the multiplicity of the second eigenvalue of the Laplacian in $\mathcal{R}_1$ is exactly two, that is,
$$
\lambda_1(2;\mathcal{R}_1)
<
\lambda_2(2;\mathcal{R}_1)
=
\lambda_3(2;\mathcal{R}_1)
<
\lambda_4(2;\mathcal{R}_1),
$$
and we have
\begin{align}
\label{eq:lrrl1}
\lambda_2(2;\mathcal{R}_1)
=
\lambda_{\boxbslashl}(2;\mathcal{R}_1)
=
\lambda_{\boxbarl}(2;\mathcal{R}_1)
=
\lambda_{\boxminus}(2;\mathcal{R}_1).
\end{align}
We would like to provide a result related to \eqref{eq:lrrl1} for the $p$-Laplacian in $\mathcal{R}_1$. 
\begin{lemma}\label{lem:lambda-derivative}
For any $p$ sufficiently close to $2$, we have
\begin{align}\label{eq:lem:lambda-derivative:1}
&\lambda_2(p;\mathcal{R}_1) 
\leq
\lambda_{\boxbarl}(p;\mathcal{R}_1)
=
\lambda_{\boxminus}(p;\mathcal{R}_1)
<
\lambda_{\boxbslashl}(p;\mathcal{R}_1)
\quad \text{if}~ p < 2,\\
\label{eq:lem:lambda-derivative:2}
&\lambda_2(p;\mathcal{R}_1) 
\leq
\lambda_{\boxbslashl}(p;\mathcal{R}_1)
<
\lambda_{\boxbarl}(p;\mathcal{R}_1)
=
\lambda_{\boxminus}(p;\mathcal{R}_1)
\quad \text{if}~ p > 2. 
\end{align} 
\end{lemma}
\begin{proof}
The first inequalities in \eqref{eq:lem:lambda-derivative:1} and \eqref{eq:lem:lambda-derivative:2} are direct. 
It is also evident that 
$$
\lambda_{\boxbarl}(p;\mathcal{R}_1)
=
\lambda_{\boxminus}(p;\mathcal{R}_1).
$$
These two facts are valid for any $p>1$. 

Let us now prove the strict inequalities in \eqref{eq:lem:lambda-derivative:1} and \eqref{eq:lem:lambda-derivative:2}, which is the main assertion. 
Note that $\lambda_2(2;\mathcal{R}_1) = 5\pi^2$. 
Let us choose the following two second eigenfunctions of the Laplacian in $\mathcal{R}_1$:
\begin{align}
\label{eq:u1}
u_1(x,y) &= \sqrt{8} \sin(2 \pi x) \sin(\pi y),\\
\label{eq:u2}
u_2(x,y) &= 2 \left(\sin(2 \pi x) \sin(\pi y) + \sin(\pi x) \sin(2\pi y) \right).
\end{align}
Clearly, $u_1$ and $u_2$ are the first eigenfunctions of the Laplacian in $\mathcal{R}_{1/2}$ and $\mathcal{T}_1$, respectively. 
The factors $\sqrt{8}$ and $2$ are chosen to make $\|u_1\|_{2;\mathcal{R}_{1/2}} = 1$ and $\|u_2\|_{2;\mathcal{T}_1} = 1$.

By Lemma~\ref{lem:differentiability:l1} and Remark~\ref{rem:differentiability:ltriangle}, $p \mapsto \lambda_1(p;\mathcal{R}_{1/2})$ and $p \mapsto \lambda_1(p;\mathcal{T}_{1})$ are differentiable at $p=2$, and their derivatives are given as 
\begin{align}\label{eq:lambdapprime2}
\left.\partial_p \lambda_1(p;\mathcal{R}_{1/2}) \right|_{p=2}
&=
\int_{\mathcal{R}_{1/2}} |\nabla u_1|^2 \ln |\nabla u_1| \,dxdy
- 
\lambda_1(p;\mathcal{R}_{1/2})
\int_{\mathcal{R}_{1/2}} u_1^2 \ln u_1 \,dxdy\\
&=
\frac{1}{2}
\int_{\mathcal{R}_{1}} |\nabla u_1|^2 \ln |\nabla u_1| \,dxdy
-
\frac{5 \pi^2}{2} 
\int_{\mathcal{R}_{1}} u_1^2 \ln |u_1| \,dxdy,
\end{align}
and, similarly, 
\begin{align}\label{eq:lambdapprime4}
\left.\partial_p \lambda_1(p;\mathcal{T}_1) \right|_{p=2}
=
\frac{1}{2} \int_{\mathcal{R}_{1}} |\nabla u_2|^2 \ln |\nabla u_2| \,dxdy
-
\frac{5 \pi^2}{2}
\int_{\mathcal{R}_{1}} u_2^2 \ln |u_2| \,dxdy.
\end{align}
We use the SciPy numerical integration routine \texttt{dblquad} to find that
\begin{equation}\label{eq:numvalues}
\left.\partial_p \lambda_1(p;\mathcal{R}_{1/2}) \right|_{p=2}
\approx
88.0203
\quad\text{and} \quad 
\left.\partial_p \lambda_1(p;\mathcal{T}_1) \right|_{p=2}
\approx
85.9285.
\end{equation}
(Although we do not provide validated upper and lower bounds for the obtained values, by the nature of integrands we consider them reliable.)
That is, the derivatives are strictly ordered. 
Recalling that 
$\lambda_1(2;\mathcal{R}_{1/2}) = \lambda_1(2;\mathcal{T}_1)$ by 
\eqref{eq:lambda=lambda}, \eqref{eq:lambda=lambda3}, 
and \eqref{eq:lrrl1}, we obtain the desired strict inequalities in \eqref{eq:lem:lambda-derivative:1} and \eqref{eq:lem:lambda-derivative:2}.
\end{proof}

\subsection{Properties of eigenvalues in rectangles}

Thanks to the continuity with respect to $a$ stated in Lemma~\ref{lem:continuity} and Corollary~\ref{cor:differentiability}, we get the following extension of Lemma~\ref{lem:lambda-derivative} to the rectangle $\mathcal{R}_a$.
\begin{lemma}\label{lem:rectangle}
For any $p>2$ sufficiently close to $2$, there exists $\sigma(p)>0$ such that for any $a \in [1-\sigma(p),1+\sigma(p)]$ we have
\begin{equation}\label{eq:lem:rectangle}
\lambda_2(p;\mathcal{R}_a)
<
\lambda_{\boxbarl}(p;\mathcal{R}_a).
\end{equation}
\end{lemma}

Since any eigenvalue of the Laplacian in the rectangle $\mathcal{R}_a$ has the form $\pi^2 \left(m^2 a^{-2} + n^2\right)$ for some $m,n \in \mathbb{N}$, one can easily derive the following result.
\begin{lemma}\label{lem:multiplicity}
Let $a \in (1,\sqrt{8/3})$. 
Then $\lambda_2(2;\mathcal{R}_a)$ and $\lambda_3(2;\mathcal{R}_a)$ are simple, and 
\begin{equation}\label{eq:lem:multiplicity}
\lambda_{\boxbarl}(2;\mathcal{R}_a)
=
\lambda_2(2;\mathcal{R}_a) 
< 
\lambda_3(2;\mathcal{R}_a)
=
\lambda_{\boxminus}(2;\mathcal{R}_a)
<
\lambda_4(2;\mathcal{R}_a).
\end{equation}
\end{lemma}

In view of Lemma~\ref{lem:multiplicity} and the continuity with respect to $p$ stated in Corollary~\ref{cor:differentiability}, we  have the following convergence result. 
\begin{lemma}\label{lem:convergence}
Let $a \in (1,\sqrt{8/3})$. 
Then
$$
\lambda_{\boxbarl}(p;\mathcal{R}_a) 
\to 
\lambda_{2}(2;\mathcal{R}_a)
\quad \text{and} \quad 
\lambda_{\boxminus}(p;\mathcal{R}_a) 
\to 
\lambda_{3}(2;\mathcal{R}_a)
\quad \text{as}~ p \to 2.
$$
\end{lemma}

By the combination of Lemmas~\ref{lem:multiplicity} and \ref{lem:convergence}, for any $a>1$ there exists $\gamma(a)>0$ such that  
$\lambda_{\boxbarl}(p;\mathcal{R}_a) 
< 
\lambda_{\boxminus}(p;\mathcal{R}_a)$ for any $p \in [2-\gamma(a),2+\gamma(a)]$. 
But such a statement is not sufficient for our purposes, cf.\ Lemma~\ref{lem:rectangle}. 
Our aim is to compare $\lambda_{\boxbarl}(p;\mathcal{R}_a)$ and $\lambda_{\boxminus}(p;\mathcal{R}_a)$ uniformly with respect to $a$ and $p$. 

\begin{lemma}\label{lem:comparison}
There exists $\gamma > 0$ such that for any $a \in (1,1+\gamma]$ and $p \in [2-\gamma,2+\gamma]$ we have 
\begin{equation}\label{eq:lem:comparison}
\lambda_{\boxbarl}(p;\mathcal{R}_a) 
< 
\lambda_{\boxminus}(p;\mathcal{R}_a).
\end{equation}
\end{lemma}
\begin{proof}
Denote
$$
f(p,a) = 
\lambda_{\boxminus}(p;\mathcal{R}_a)
-
\lambda_{\boxbarl}(p;\mathcal{R}_a). 
$$
We know from Corollary~\ref{cor:differentiability} that $f \in \mathcal{C}^1((1,+\infty) \times (0,+\infty))$. 
It is evident from \eqref{eq:lambda=lambda} and \eqref{eq:lambda=lambda2} that $f(p,1)=0$ for any $p>1$. 
Let us calculate $\partial_a f(2,a)$.  
Using \eqref{eq:lambda=lambda} and \eqref{eq:lambda=lambda2}, and the explicit values of the Laplace eigenvalues in rectangles, we have
$$
\lambda_{\boxbarl}(2;\mathcal{R}_a) = 
\pi^2 \left(4 a^{-2}+1\right)
\quad \text{and} \quad 
\lambda_{\boxminus}(2;\mathcal{R}_a) = 
\pi^2 \left(a^{-2}+4\right).
$$
Therefore, 
\begin{equation}\label{eq:partiala}
\partial_a \lambda_{\boxbarl}(2;\mathcal{R}_a) 
= 
-8 \pi^2 a^{-3}
\quad \text{and} \quad
\partial_a \lambda_{\boxminus}(2;\mathcal{R}_a) 
= 
-2 \pi^2 a^{-3},
\end{equation}
and hence $\partial_a f(2,a) = 6 \pi^2 a^{-3} > 0$. 
By the continuity, $\partial_a f(p,a)$ stays positive for any $p$ sufficiently close to $2$ and any $a$ sufficiently close to $1$. 
Thus, for such $p$ and $a>1$, 
$$
f(p,a)
=
f(p,a) - f(p,1) = \int_1^a \partial_a f(p,t) \,dt > 0,
$$
which completes the proof. 
\end{proof}

\begin{remark}
Lemma~\ref{lem:comparison} would be a consequence of the monotonicity (or convexity) of the first eigenvalue of the $p$-Laplacian over the class of equimeasurable rectangles. 
However, in the case $p \neq 2$, we were not able to find a corresponding result in the literature. 
\end{remark}

\subsection{Uniform spectral gap}

\begin{lemma}\label{lem:gap}
Let $A \in (0,3 \pi^2/2]$. 
Then 
there exists $\delta>0$ such that for any $a \in [1-\delta,1+\delta]$ and $p \in [2-\delta,2+\delta]$ there are no eigenvalues of the $p$-Laplacian in $\mathcal{R}_a$ belonging to the interval 
$[\lambda_3(2;\mathcal{R}_1)+A, \lambda_4(2;\mathcal{R}_1)-A] \neq \emptyset$,
and we have
\begin{equation}\label{eq:lem:gap}
\lambda_3(p;\mathcal{R}_a)
<
\lambda_3(2;\mathcal{R}_1) + A
\leq 
\lambda_4(2;\mathcal{R}_1) - A
<
\lambda_4(p;\mathcal{R}_a).
\end{equation}
\end{lemma}

Note that $\lambda_3(2;\mathcal{R}_1) = 5\pi^2$ and $\lambda_4(2;\mathcal{R}_1)=8\pi^2$, so that the interval $[\lambda_3(2;\mathcal{R}_1)+A, \lambda_4(2;\mathcal{R}_1)-A]$ is nonempty whenever $A \in (0,3 \pi^2/2]$. 

\begin{proof}[Proof of Lemma~\ref{lem:gap}]
Suppose, by contradiction, that there is some $A \in (0,3 \pi^2/2]$ and sequences $a_n \to 1$ and $p_n \to 2$ such that for each $n$ there exists an eigenvalue 
$$
\lambda(p_n;\mathcal{R}_{a_n}) \in [\lambda_3(2;\mathcal{R}_1)+A, \lambda_4(2;\mathcal{R}_1)-A].
$$  
Let $u_n \in W_0^{1,p_n}(\mathcal{R}_{a_n})$ be an eigenfunction associated with $\lambda(p_n;\mathcal{R}_{a_n})$ and normalized as $\|u_n\|_{p_n;\mathcal{R}_{a_n}} = 1$. 
Without loss of generality, we can assume that $\{\lambda(p_n;\mathcal{R}_{a_n})\}$ converges to some $\lambda \in [\lambda_3(2;\mathcal{R}_1)+A, \lambda_4(2;\mathcal{R}_1)-A]$. 
Thanks to Lemma~\ref{lem:compactness}, 
$\{u_n\}$ converges to an eigenfunction of the Laplacian in $\mathcal{R}_1$ whose eigenvalue $\lambda$ lies strictly between $\lambda_3(2;\mathcal{R}_1)$ and $\lambda_4(2;\mathcal{R}_1)$, which is impossible. 

Recalling from Lemma~\ref{lem:continuity} that the mappings $(p,a) \mapsto \lambda_3(p;\mathcal{R}_a)$ and $(p,a) \mapsto \lambda_4(p;\mathcal{R}_a)$ are continuous, we obtain \eqref{eq:lem:gap}. 
\end{proof}

\section{Proof and discussion}\label{sec:proof}

\begin{proof}[Proof of Theorem~\ref{thm:main}]
Let us fix any $A \in (0,3 \pi^2/2]$, so that the interval
$$
I 
:= 
[\lambda_3(2;\mathcal{R}_1)+A, \lambda_4(2;\mathcal{R}_1)-A]
$$
is nonempty. 
Let us fix a sufficiently small $p_0 \in (2,2+\min\{\delta,\gamma\}]$ and any $a \in (1,1+\min\{\delta,\sigma(p_0),\gamma,\sqrt{8/3}-1\}]$, where
$\sigma(p_0), \gamma, \delta>0$ are given by Lemmas~\ref{lem:rectangle}, \ref{lem:comparison}, \ref{lem:gap}, respectively. 
In particular, for such $p_0$ and $a$ we have
\begin{gather}
\label{eq:proof:1}
\lambda_2(p_0;\mathcal{R}_a)
<
\lambda_{\boxbarl}(p_0;\mathcal{R}_a),\\
\label{eq:proof:-1}
\lambda_{\boxbarl}(p_0;\mathcal{R}_a)
<
\lambda_{\boxminus}(p_0;\mathcal{R}_a),\\
\label{eq:proof:0}
\lambda_3(2;\mathcal{R}_a)
<
\lambda_3(2;\mathcal{R}_1) + A
\leq 
\lambda_4(2;\mathcal{R}_1) - A
<
\lambda_4(p_0;\mathcal{R}_a).
\end{gather}

Let us now prove that at least one of $\lambda_{\boxbarl}(p_0;\mathcal{R}_a)$ and  $\lambda_{\boxminus}(p_0;\mathcal{R}_a)$ is a non-LS eigenvalue. 
Suppose, by contradiction, that both
$\lambda_{\boxbarl}(p_0;\mathcal{R}_a)$ and $\lambda_{\boxminus}(p_0;\mathcal{R}_a)$ are LS eigenvalues. 
By \eqref{eq:proof:1}, we have
$$
\lambda_{\boxbarl}(p_0;\mathcal{R}_a)
=
\lambda_{k}(p_0;\mathcal{R}_a) \quad \text{for some}~ k \geq 3.
$$
In view of \eqref{eq:proof:-1}, we further have
$$
\lambda_{\boxminus}(p_0;\mathcal{R}_a)
=
\lambda_{m}(p_0;\mathcal{R}_a) \quad \text{for some}~ m \geq k+1 \geq 4.
$$
Thus, \eqref{eq:proof:0} yields
$$
\lambda_3(2;\mathcal{R}_a)
<
\lambda_3(2;\mathcal{R}_1) + A
\leq 
\lambda_4(2;\mathcal{R}_1) - A
<
\lambda_{\boxminus}(p_0;\mathcal{R}_a).
$$
Since the mapping $p \to \lambda_{\boxminus}(p;\mathcal{R}_a)$ is continuous and 
$\lambda_{\boxminus}(p;\mathcal{R}_a) \to \lambda_3(2;\mathcal{R}_a)$ as $p \to 2$ by Lemma~\ref{lem:convergence}, we deduce the existence of $p_1 \in (2,p_0)$ such that $\lambda_{\boxminus}(p_1;\mathcal{R}_a) \in I$. 
But this contradicts the nonexistence given by Lemma~\ref{lem:gap}.
\end{proof}

The result of Theorem~\ref{thm:main} can be slightly strengthen by means of the following stability assertion. 
\begin{proposition}\label{prop:pert}
Let $\lambda_{\boxbarl}(p;\mathcal{R}_{a})$ (resp., $\lambda_{\boxminus}(p;\mathcal{R}_{a})$) be a non-LS eigenvalue for some $p>1$ and $a>0$. 
Then there exists $\varepsilon>0$ such that $\lambda_{\boxbarl}(q;\mathcal{R}_b)$ (resp., $\lambda_{\boxminus}(q;\mathcal{R}_b)$) is a non-LS eigenvalue for any $q \in (p-\varepsilon, p+\varepsilon)$ and $b \in (a-\varepsilon, a+\varepsilon)$.
\end{proposition}
\begin{proof}
Suppose, by contradiction, that there are sequences $p_n \to p$ and $a_n \to a$ such that each $\lambda_{\boxbarl}(p_n;\mathcal{R}_{a_n})$ is an LS eigenvalue, but $\lambda_{\boxbarl}(p;\mathcal{R}_{a})$ is not. 
That is, for any $n \in \mathbb{N}$ there exists $k(n) \in \mathbb{N}$ such that 
$$
\lambda_{\boxbarl}(p_n;\mathcal{R}_{a_n})
=
\lambda_{k(n)}(p_n;\mathcal{R}_{a_n}).
$$
Since $\{k(n)\}$ is bounded in view of \eqref{eq:sequenceoflambdas}, and due to the continuity of LS eigenvalues given by Lemma~\ref{lem:continuity}, we can extract a subsequence along which $\{k(n)\}$ converges to some $k \in \mathbb{N}$. 
But then Lemma~\ref{lem:continuity} and Corollary~\ref{cor:differentiability} yield $\lambda_{\boxbarl}(p;\mathcal{R}_{a})=\lambda_{k}(p;\mathcal{R}_{a})$, which contradicts our assumption. The same arguments apply to $\lambda_{\boxminus}(p;\mathcal{R}_{a})$.
\end{proof}

We finish this section with a few general remarks. 

\begin{remark}
Theorem~\ref{thm:main} resolves in negative the second conjecture of \textsc{Friedlander} \cite[p.~1069]{F} asking, in essence, whether any LS eigenvalue of the $p$-Laplacian in $\mathcal{R}_a$ has an eigenfunction obtained by the antisymmetric extension to $\mathcal{R}_a$ of the first eigenfunction in a subrectangle (which is a generalization of the construction of $\lambda_{\boxbarl}(p;\mathcal{R}_{a})$ and $\lambda_{\boxminus}(p;\mathcal{R}_{a})$).
Information on the spectrum of the $p$-Laplacian in $\mathcal{R}_a$ is required for deriving the Weyl asymptotics of $\{\lambda_k(p;\Omega)\}$, and we refer to \cite{bob-plej} for a further discussion. 
\end{remark}

\begin{remark}\label{reg:gen} 
Apart from the sequence $\{\lambda_k(p;\mathcal{R}_a)\}$ constructed via the Krasnoselskii genus, the proof of Theorem~\ref{thm:main} remains valid with no changes for any nondecreasing sequence of eigenvalues $\{\mu_k(p;\mathcal{R}_a)\}$ of the $p$-Laplacian in $\mathcal{R}_a$ satisfying the following three properties in a neighborhood of $(p,a)=(2,1)$: 
\begin{enumerate}
\item $\mu_k(2;\mathcal{R}_a) = \lambda_k(2;\mathcal{R}_a)$ for $k \in \{2,3,4\}$;
\item $(p,a) \mapsto \mu_k(p;\mathcal{R}_a)$ is continuous for $k \in \{2,3,4\}$;
\item $\mu_2(p;\mathcal{R}_a) = \lambda_2(p;\mathcal{R}_a)$.
\end{enumerate}
As examples, the sequences defined via the cogenus and cohomological index do satisfy these assumptions (cf.\ \cite{bob-plej}), and hence they do not exhaust the spectrum of the $p$-Laplacian in $\mathcal{R}_a$.
\end{remark}

\begin{remark}\label{rem:variational} 
Although $\lambda_{\boxbarl}(p;\mathcal{R}_{a})$ (resp., $\lambda_{\boxminus}(p;\mathcal{R}_{a})$) might be a non-LS eigenvalue for some $p>1$ and $a>0$, it is \textit{variational} in the sense that it can be characterized as the global minimum of the Rayleigh quotient over the subspace of $W_0^{1,p}(\mathcal{R}_{a})$ consisting of functions antisymmetric with respect to the vertical (resp., horizontal) middle line of $\mathcal{R}_{a}$. 
We refer to \cite{drabek-var,drabek-takac} for a related discussion. 
\end{remark}

\begin{remark}
When $\lambda_{\boxbarl}(p;\mathcal{R}_{a})$ (resp., $\lambda_{\boxminus}(p;\mathcal{R}_{a})$) is a non-LS eigenvalue for some $p>1$ and $a>0$, it is interesting to know whether this eigenvalue is stable with respect to nonsymmetric domain perturbations (cf.\ Proposition~\ref{prop:pert}). Namely, for any $\varepsilon>0$ and any sufficiently small (in an appropriate sense) perturbation $\Omega$ of $\mathcal{R}_{a}$ destroying symmetry, is there an eigenvalue of the $p$-Laplacian in $\Omega$ lying in the $\varepsilon$-neighborhood of the level $\lambda_{\boxbarl}(p;\mathcal{R}_{a})$ (resp., $\lambda_{\boxminus}(p;\mathcal{R}_{a})$)?
\end{remark}

\begin{remark}\label{rem:horak}
In \cite[Sections~4.2 and 4.3]{Horak}, \textsc{Hor\'ak} computed several lower energy eigenvalues of the $p$-Laplacian in the square and rectangles and formulated the following conjecture related to our analysis: 
For any $a>1$, there exists $p_*(a)>2$ such that 
\begin{align}
\lambda_2(p;\mathcal{R}_a)
&=
\lambda_{\boxbarl}(p;\mathcal{R}_a)
\quad \text{for any}~ p \in [2,p_*(a)],\\
\lambda_2(p;\mathcal{R}_a)
&<
\lambda_{\boxbarl}(p;\mathcal{R}_a)
\quad \text{for any}~ p > p_*(a).
\end{align}
If this is true, then $p_*(a) \to 2$ as $a \to 1$ by Lemma~\ref{lem:rectangle}. 
We further conjecture that $p_*(a) = +\infty$ for any $a \geq 2$. 
\end{remark}

\section{On the third LS eigenvalue in the square}\label{sec:abitmoreaboutsquare}

For the $p$-Laplacian in the square $\mathcal{R}_1$, it was conjectured by \textsc{Kawohl} \cite{kawohl} with the numerical support of \cite{Horak} and 
the analysis of the limiting cases $p \to 1$ and $p \to +\infty$ that
\begin{align}\label{eq:square1}
&\lambda_2(p;\mathcal{R}_1)
=
\lambda_{\boxbarl}(p;\mathcal{R}_1)
<
\lambda_{\boxbslashl}(p;\mathcal{R}_1)
\quad \text{for any}~ p \in (1,2),\\
\label{eq:square2}
&\lambda_2(p;\mathcal{R}_1)
=
\lambda_{\boxbslashl}(p;\mathcal{R}_1)
<
\lambda_{\boxbarl}(p;\mathcal{R}_1)
\quad \text{for any}~ p > 2,
\end{align} 
cf.\ \eqref{eq:lrrl1} and Remark~\ref{rem:horak}. 
Our Lemma~\ref{lem:lambda-derivative} further supports this conjecture for $p$ sufficiently close to~$2$. 

In the linear case $p=2$, we plainly have
$$
\lambda_2(2;\mathcal{R}_1)
=
\lambda_{\boxbarl}(2;\mathcal{R}_1)
=
\lambda_{\boxbslashl}(2;\mathcal{R}_1)
=
\lambda_3(2;\mathcal{R}_1).
$$
Yet, for $p \neq 2$, we expect $\lambda_2(p;\mathcal{R}_1) < \lambda_3(p;\mathcal{R}_1)$. 
This expectation is partially motivated by the fact that equality of these eigenvalues would imply the Krasnoselskii genus of the set of normalized second eigenfunctions to be at least two, yielding a rich family of second eigenfunctions beyond possible ``symmetric'' eigenfunctions. 
When $p=2$, such a family comes from the superposition principle, but in the nonlinear case the superposition principle is no longer available. 
Based on this expectation, it might be tempting to conjecture that, in addition to \eqref{eq:square1} and \eqref{eq:square2}, the following equalities hold:
\begin{align}\label{eq:square3}
&\lambda_{\boxbslashl}(p;\mathcal{R}_1)
=
\lambda_3(p;\mathcal{R}_1)
\quad \text{for any}~ p \in (1,2),\\
\label{eq:square4}
&\lambda_{\boxbarl}(p;\mathcal{R}_1)
=
\lambda_3(p;\mathcal{R}_1)
\quad \text{for any}~ p > 2. 
\end{align} 
The situation, however, is more intriguing, as \eqref{eq:square4} cannot be true for $p$ large enough.
\begin{theorem}
There exists $\hat{p} \geq 2$ such that 
\begin{equation}\label{eq:square5}
\lambda_3(p;\mathcal{R}_1)
<
\lambda_{\boxbarl}(p;\mathcal{R}_1)
\quad \text{for any}~ p > \hat{p}. 
\end{equation}
\end{theorem}
\begin{proof}
Thanks to \cite[Lemma~1.5]{juutinen} and \eqref{eq:lambda=lambda}, we have
\begin{equation}\label{eq:square0}
\lambda_{\boxbarl}^{1/p}(p;\mathcal{R}_1)
=
\lambda_{1}^{1/p}(p;\mathcal{R}_{1/2})
\to 
\frac{1}{r(\mathcal{R}_{1/2})}
=
4
\quad \text{as}~ p \to +\infty,
\end{equation}
where by $r(\Omega)$ we denote the inradius of a domain $\Omega$. 

\begin{figure}[!ht]
\centering
\resizebox{0.3\linewidth}{!}{%
\begin{tikzpicture}[x=7cm,y=7cm,line width=0.7pt]
  \pgfmathsetmacro{\r}{2/(4+sqrt(6)+sqrt(2))}
  \pgfmathsetmacro{\a}{\r+(1-2*\r)*(2-sqrt(3))}

  \draw[fill=gray!10] (\r,\r) circle[radius=\r];
  \draw[fill=gray!10] ({1-\r},\a) circle[radius=\r];
  \draw[fill=gray!10] (\a,{1-\r}) circle[radius=\r];
  \draw (0,0) rectangle (1,1);
  
  \node at (\r,\r) {$B_1$};
  \node at ({1-\r},\a) {$B_2$};
  \node at (\a,{1-\r}) {$B_3$};
\end{tikzpicture}
}
\caption{Three disks packed in $\mathcal{R}_1$.}
\label{fig2}
\end{figure}

It is known that the maximal radius of three disjoint equal disks packed in $\mathcal{R}_1$ is $(2+\sqrt{2}/2+\sqrt{6}/2)^{-1} = 0.2543...$, see, e.g, \cite{goldberg}.
Let us denote such disks as $B_1$, $B_2$, $B_3$ (cf.\ Figure~\ref{fig2}), and let 
$$
\delta_j(x) = 
\begin{cases}
\text{dist}(x,\partial B_j) &\text{if}~ x \in B_j,\\
0 &\text{if}~ x \in \mathcal{R}_1 \setminus \overline{B_j},
\end{cases}
\qquad j=1,2,3.
$$
We have $|\nabla \delta_j| = 1$ a.e.\ in $B_j$, and $|\nabla \delta_j| = 0$ in $\mathcal{R}_1 \setminus B_j$. 
Moreover, for any constants $c_1,c_2,c_3 \in \mathbb{R}$, the function $c_1 \delta_1 + c_2 \delta_2 + c_3 \delta_3$ is Lipschitz and belongs to $W_0^{1,p}(\mathcal{R}_1)$ for any $p>1$. 
Consider a set $\mathcal{A} \subset W_0^{1,p}(\mathcal{R}_1) \setminus \{0\}$ defined as 
$$
\mathcal{A} = \{v:~ v = c_1 \delta_1 + c_2 \delta_2 + c_3 \delta_3 ~\text{for some}~ (c_1,c_2,c_3) \in \mathbb{R}^3 \setminus \{0\},
~\text{and}~ \|v\|_{p;\mathcal{R}_1}=1\}.
$$
This set is symmetric and compact.
The mapping $f: \mathcal{A} \mapsto \mathbb{R}^3 \setminus \{0\}$ defined as
$f(c_1 \delta_1 + c_2 \delta_2 + c_3 \delta_3) = (c_1,c_2,c_3)$
is continuous and odd homeomorphism to its image, and hence $\gamma_K(\mathcal{A}) = 3$. 
It is then not hard to calculate that 
\begin{align}\label{eq:square0x}
\lambda_3^{1/p}(p;\mathcal{R}_1) \leq \sup_{v \in \mathcal{A}} R_p^{1/p}[v;\mathcal{R}_1]
=
\frac{1}{\left(\frac{1}{|B_j|} \int_{B_j} \delta_j^p \,dxdy \right)^{1/p}} 
\to \frac{1}{r(B_j)} = 3.9318...
\quad \text{as}~ p \to +\infty.~~
\end{align}
Comparing \eqref{eq:square0} and \eqref{eq:square0x}, we obtain the desired inequality \eqref{eq:square5}. 
\end{proof}

\bigskip

\noindent
\textbf{Acknowledgments.} 
The author is thankful to P.~Dr\'abek for a discussion on the history of the subject. 
The author is deeply grateful to L.~Brasco and A.I.~Nazarov for their interest in this work and for valuable suggestions leading to simplification and clarification of the arguments, to stronger statements, and to the improvement of exposition.

    \bigskip
	\noindent
	\textbf{Statements and Declarations.}
 The author declares no potential conflict of interest.
 The work is fully AI-free: no AI system was involved in thinking of ideas and proofs, and in writing the present text.

\phantomsection


\addcontentsline{toc}{section}{\refname}
\small
\begin{thebibliography}{99}

\bibitem{amann}
Amann, H. (1972). Lusternik-Schnirelman theory and non-linear eigenvalue problems. Mathematische Annalen, 199(1), 55-72.
\doi{10.1007/BF01419576}

\bibitem{anane}
Anane, A. (1988). Etude des valeurs propres et de la résonance pour l'opérateur $p$-Laplacien. PhD thesis.

\bibitem{ananechakrone}
Anane, A., \& Chakrone, O. (2000). Sur un th\'eor\`eme de point critique et application \`a un probl\`eme de non-r\'esonance entre deux valeurs propres du $p$-laplacien. Annales de la Facult\'e des Sciences de Toulouse. Math\'ematiques. S\'erie 6, 9(1), 5-30.
\url{https://www.numdam.org/item/AFST_2000_6_9_1_5_0/}

\bibitem{ADS}
Anoop, T. V., Dr\'abek, P., \& Sasi, S. (2016). On the structure of the second eigenfunctions of the $p$-Laplacian on a ball. Proceedings of the American Mathematical Society, 144(6), 2503-2512.
\doi{10.1090/proc/12902}

\bibitem{barbatis-lamberti}
Barbatis, G., \& Lamberti, P. D. (2016). Monotonicity, continuity and differentiability results for the $L^p$ Hardy constant. Israel Journal of Mathematics, 215(2), 1011-1024.	
\doi{10.1007/s11856-016-1400-z}

\bibitem{bindingrynne1}	
Binding, P. A., \& Rynne, B. P. (2008). Variational and non-variational eigenvalues of the $p$-Laplacian. Journal of Differential Equations, 244(1), 24-39.
\doi{10.1016/j.jde.2007.10.010}

\bibitem{bob-plej}
Bobkov, V. (2026). On Pleijel-type nodal domain bounds for the $p$-Laplacian. 
\href{arXiv:2606.31305}{https://arxiv.org/abs/2606.31305}

\bibitem{brasco-franzina-pathalogical}
Brasco, L., \& Franzina, G. (2017). A pathological example in nonlinear spectral theory. Advances in Nonlinear Analysis, 8(1), 707-714.
\doi{10.1515/anona-2017-0043}

\bibitem{brasco-parini-squassina}
Brasco, L., Parini, E., \& Squassina, M. (2016). Stability of variational eigenvalues for the fractional $p$-Laplacian. Discrete and Continuous Dynamical Systems, 36(4), 1813-1845
\doi{10.3934/dcds.2016.36.1813}

\bibitem{browder1}
Browder, F.E. (1970). Nonlinear eigenvalue problems and group invariance. 
In Functional Analysis and Related Fields: Proceedings of a Conference in honor of Professor Marshall Stone, held at the University of Chicago, May 1968 (pp. 1-58). Berlin, Heidelberg: Springer.
\doi{10.1007/978-3-642-48272-4\_1}

\bibitem{champione}
Champion, T., \& De Pascale, L. (2007). Asymptotic behaviour of nonlinear eigenvalue problems involving $p$-Laplacian-type operators. Proceedings of the Royal Society of Edinburgh Section A: Mathematics, 137(6), 1179-1195.
\doi{10.1017/S0308210506000667}

\bibitem{cuesta}
Cuesta, M. (2000). On the Fu\v{c}\'ik spectrum of the Laplacian and $p$-Laplacian.
In Proceedings of the ``2000 Seminar in Differential Equations'', Kvilda (Czech Republic), 2000.
\url{https://www.academia.edu/download/102228240/kavilda0.pdf}

\bibitem{degiovannimarzocchi}
Degiovanni, M., \& Marzocchi, M. (2015). 
On the dependence on $p$ of the variational eigenvalues of the $p$-Laplace operator. Potential Analysis, 43(4), 593-609.
\doi{10.1007/s11118-015-9487-0}

\bibitem{diblasio-lamberti}
Di Blasio, G., \& Lamberti, P. D. (2020). Eigenvalues of the Finsler $p$‐Laplacian on varying domains. Mathematika, 66(3), 765-776.
\doi{10.1112/mtk.12042}

\bibitem{drabek-var}
Dr\'abek, P. (2012). On the variational eigenvalues which are not of Ljusternik‐Schnirelmann type. 
Abstract and Applied Analysis, 2012(1), 434631.
\doi{10.1155/2012/434631}

\bibitem{drabek-robinson}
Dr\'abek, P., \& Robinson, S. (2000). Resonance problems for the one-dimensional $p$-Laplacian. Proceedings of the American Mathematical Society, 128(3), 755-765.
\doi{10.1090/S0002-9939-99-05485-4}

\bibitem{drabek-takac}
Dr\'abek, P., \& Tak\'a\v{c}, P. (2010). On variational eigenvalues of the $p$-Laplacian which are not of Ljusternik–Schnirelmann type. Journal of the London Mathematical Society, 81(3), 625-649.
\doi{10.1112/jlms/jdq006}

\bibitem{F} 
Friedlander, L. (1989). Asymptotic behaviour of the eigenvalues of the $p$-Laplacian. Communications in Partial Differential Equations, 14(8-9), 1059-1069.
\doi{10.1080/03605308908820643}

\bibitem{FNSS2}
Fu\v{c}\'ik, S., Ne\v{c}as, J., Sou\v{c}ek, J., \& Sou\v{c}ek, V. (1972). Upper bound for the number of critical levels for nonlinear operators in Banach spaces of the type of second order nonlinear partial differential operators. Journal of Functional Analysis, 11(3), 314-333.
\doi{10.1016/0022-1236(72)90072-9}

\bibitem{FNSS}
Fu\v{c}\'ik, S., Ne\v{c}as, J., Sou\v{c}ek, J., \& Sou\v{c}ek, V. (1973). Spectral analysis of nonlinear operators. Springer. 
\doi{10.1007/BFb0059360}

\bibitem{garciaperal}
Garc\'ia Azorero, J.P., \& Peral Alonso, I. (1987). Existence and nonuniqueness for the $p$-Laplacian. Communications in Partial Differential Equations, 12(12), 126-202.
\doi{10.1080/03605308708820534}

\bibitem{garcia}
Garc\'ia Meli\'an, J., \& Sabina de Lis, J. (2001). On the perturbation of eigenvalues for the $p$-Laplacian. Comptes Rendus de l'Acad\'emie des Sciences-Series I-Mathematics, 332(10), 893-898.
\doi{10.1016/s0764-4442(01)01956-5}

\bibitem{goldberg}
Goldberg, M. (1970). The packing of equal circles in a square. Mathematics Magazine, 43(1), 24-30.
\doi{10.1080/0025570X.1970.11975991}

\bibitem{Horak}
Hor\'ak J. (2011). Numerical investigation of the smallest eigenvalues of the $p$-Laplace operator on planar domains. Electronic Journal of Differential Equations, 2011(132), 1-30. 
\url{https://ejde.math.txstate.edu/Volumes/2011/132/horak.pdf}

\bibitem{juutinen}
Juutinen, P., Lindqvist, P., \& Manfredi, J. J. (1999). The $\infty$-eigenvalue problem. 
Archive for rational mechanics and analysis, 148(2), 89-105.
\doi{10.1007/s002050050157}

\bibitem{kawohl}
Kawohl, B. (2011). Variations on the $p$-Laplacian. Nonlinear Elliptic Partial Differential Equations, Eds. D. Bonheure, P. Tak\'a\v{c} et al., Contemporary Mathematics, 540, 35-46.
\doi{10.1090/conm/540}

\bibitem{kawohlling}
Kawohl, B., \& Lindqvist, P. (2006). Positive eigenfunctions for the $p$-Laplace operator revisited. Analysis-International Mathematical Journal of Analysis and its Application, 26(4), 545.
\doi{10.1524/anly.2006.26.4.545}

\bibitem{lind}
Lindqvist, P. (1993). On non-linear Rayleigh quotients. Potential Analysis, 2(3), 199-218.
\doi{10.1007/BF01048505}

\bibitem{parini-continuity}
Parini, E. (2011). Continuity of the variational eigenvalues of the $p$-Laplacian with respect to $p$. Bulletin of the Australian Mathematical Society, 83(3), 376-381.
\doi{10.1017/S000497271100205X}

\bibitem{szulkin}
Szulkin, A. (1988). Ljusternik--Schnirelmann theory on $C^1$-manifolds. Annales de l'Institut Henri Poincar\'e C, Analyse non lin\'eaire, 5(2), 119-139.
\doi{10.1016/S0294-1449(16)30348-1}

\bibitem{tolksdorf}
Tolksdorf, P. (1984). Regularity for a more general class of quasilinear elliptic equations. Journal of Differential equations, 51(1), 126-150. 
\doi{10.1016/0022-0396(84)90105-0}

\end{thebibliography}
\end{document}